\documentclass [a4paper,12pt] {article}
\usepackage[utf8]{inputenc}
\usepackage[english]{babel}
\usepackage{amsmath,amsthm,amssymb}
\usepackage{cmap}
\usepackage{mathtools}
\usepackage{mathtext}
\usepackage{graphicx}
\usepackage{comment}

\usepackage[bookmarks]{hyperref}
\date{}

\newtheorem{theorem}{Theorem}[section]
\newtheorem{lemma}[theorem]{Lemma}

\theoremstyle{definition}

\newtheorem{rmk}[theorem]{Remark}
\newtheorem{defin}[theorem]{Definition}

\newtheorem{ex}[theorem]{Example}

\def\F{{\mathbb F}}

\def\Z{{\mathbb Z}}

\newcommand{\diam}{\mathrm{diam\,}}

\newcommand{\dist}{d}

\newcommand{\E}{\mathrm{E\,}}
\newcommand{\V}{\mathrm{V\,}}

\renewcommand{\ge}{\geqslant}
\renewcommand{\le}{\leqslant}

\begin{document}

\title{Orthogonality graph of the algebra of upper triangular matrices\thanks{This work is partially  financially supported by grants  RFBR  15-01-01132, RFBR 15-31-20329  and  MD-962.2014.01 } }
\author{B.R. Bakhadly\\ Lomonosov Moscow State University, \\ Faculty of Mechanics and Mathematics}
\date{}

\maketitle

\begin{abstract}
We study the connectedness and the diameter of orthogonality graphs of upper triangular matrix algebras over arbitrary fields.


\end{abstract}

\section{Introduction}

Nowadays binary relations on associative rings, in particular, on the matrix algebra can be investigated with the help of graph theory if we study the so-called {\it relation graph} whose vertices are the elements of some set and two vertices are connected by an edge if and only if the corresponding elements are in this relation. Commuting graphs and zero divisor graphs are the examples of relation graphs that have been studied intensively during the last $20$ years. This article concerns orthogonality graphs.
In the paper~\cite{BGM} the notion of graph generated by the mutual orthogonality relation for the elements of an associative ring was introduced. The authors of~\cite{BGM} computed the diameters of orthogonality graphs of the full matrix algebra over an arbitrary field and its subsets consisting of diagonal, diagonalizable, triangularizable, nilpotent, and niltriangular matrices.
The relation of orthogonality can be found in~\cite{a31, Ovchin, Semrl} where some partial orders on matrix algebra and matrix transformations which are monotone with respect to these orders are studied. Matrix orders are widely used in various fields of algebra and have applications in mathematical statistics and many other areas of mathematics~\cite{a32}. For the detailed and self-contained information on this topic see~\cite{a2,a4,a3,a5,BGM} and the references therein.

The main subject of our research is connected with triangular matrices. In~\cite{AR} Akbari and Raja proved that if $n \ge 2$ and $U$ is the set of all upper triangular matrices, then for every algebraic division ring $D$ the commuting graph of $U$ is connected.
Besides, the zero divisor graph of upper triangular matrices over commutative rings was also investigated in papers~\cite{k1,k2,k3}.
The aim of this paper is to prove the connectedness and calculate the diameter of orthogonality graphs of upper triangular matrix algebras over arbitrary fields.

\medskip

Recall some definitions from   graph theory. The notions of  graph theory used in this article can be found for example in~\cite[Chapter 2]{Harari}.

A {\em graph\/} $\Gamma$ is a non-empty set of vertices $\V(\Gamma)$ and edges $\E(\Gamma)$.
If $v_1$, $v_2$ are two vertices and $e=(v_1,v_2)$ is the edge connecting them, then the vertex $v_1$ and the edge $e$ are {\em incident\/},
the vertex $v_2$ and the edge $e$ are also {\em incident\/}.
A {\em path (walk)\/} is a sequence of vertices and edges $v_0, e_1, v_1, e_2, v_2, \ldots , e_k, v_k$,
where any two neighbor elements are incident.
If $v_0=v_k$, then the path is {\em closed\/}.
The {\em length of a path\/}, denoted by $d$,  is the number of edges that it uses, under the condition that each edge is counted as many times as it occurs in the path.
For the path $M = v_0, e_1, v_1, e_2, v_2, \ldots, e_k, v_k$ the length of $M$ equals~$k$, regardless of whether the edges are repeated or not in the path.
The graph is said to be {\em connected\/}
if it is possible to establish a path from any vertex to any other vertex of the graph.
The {\em distance $\dist(u,v)$\/} between two vertices~$u$ and~$v$ in a~graph $\Gamma$ is the length of the shortest path between them.
If~$u$ and~$v$ are unreachable from each other, $\dist(u,v)=\infty$. It is assumed that $\dist(u,u)=0$ for any vertex $u$.
The {\em diameter $\diam(\Gamma)$\/} of a graph $\Gamma$ is the maximum of distances between vertices for all pairs of vertices in the graph.

Let us introduce some notations that will be needed in this paper.
Throughout our paper, $\F$~and $R$ denote an arbitrary field and an arbitrary associative ring with unity, respectively.
$M_{m,n}\left(\F\right)$ is the set of $m\times n$ matrices over $\F$,
$ M_{n}\left(\F\right)=M_{n,n}(\F)$ is the ring (or algebra) of  $n \times n$ matrices over $\F$,
$GL_{n}\left(\F\right)$ is the group of invertible matrices in $M_n(\F)$.
$T_n$ denotes the set of all upper triangular matrices in $M_n(\F)$.
$E_{ij}$ (or $E_{i,j}$) is a matrix whose $(i,j)$-entry is~1 and other entries are~0.
If $A$ is a matrix, then $A^t$ denotes the transpose of $A$.
It is considered that $\F^n=M_{n,1}(\F)$.
$0_{n \times m}$ and $0_n$ are zero matrices of sizes $n \times m$ and $n\times n$, respectively.
$I_r$ is the identity matrix of size $r \times r$ and $J_r$ is the $r\times r$ Jordan block with the eigenvalue $0$.

The following definition is well-known.
\begin{defin}
Two elements $r_1 \in R$ and $r_2 \in R$ are {\em orthogonal\/} if $r_1r_2 = r_2r_1 = 0$.

$O_{R}\left(X\right)$ denotes the set of elements of $R$ which are orthogonal to all elements of $X$, where $X$ is a subset of $R$.
\end{defin}

The definition of orthogonality graphs was introduced and investigated by the present author, Guterman, and Markova in~\cite{BGM}.
\begin{defin}[{\cite[Definition~2.15]{BGM}}]
With every ring $R$ one can associate the \textit{orthogonality graph} $O\left(R\right)$ with vertex set consisting of all non-zero two-sided zero divisors of $R$, in which two vertices are connected by an edge if and only if the corresponding elements of $R$ are orthogonal.
\end{defin}

\begin{ex}
Let $R = M_2(\Z_2)$.
Denote $A_{01} = \begin{pmatrix}
1 & 1 \\
1 & 1
\end{pmatrix},$
$A_{11} = \begin{pmatrix}
1 & 0 \\
1 & 1
\end{pmatrix},$
$A_{12} =\begin{pmatrix}
0 & 1 \\
1 & 1
\end{pmatrix},$
$A_{13} = \begin{pmatrix}
1 & 1 \\
0 & 1
\end{pmatrix},$
$A_{14} = \begin{pmatrix}
1 & 1 \\
1 & 0
\end{pmatrix},$
$A_{21} = \begin{pmatrix}
0 & 0 \\
1 & 1
\end{pmatrix},$
$A_{22} = \begin{pmatrix}
1 & 1 \\
0 & 0
\end{pmatrix},$
$A_{23} = \begin{pmatrix}
0 & 1 \\
0 & 1
\end{pmatrix},$
$A_{24}~=~\begin{pmatrix}
1 & 0 \\
1 & 0
\end{pmatrix},$
$A_{25} = \begin{pmatrix}
0 & 1 \\
1 & 0
\end{pmatrix},$
$A_{31} = \begin{pmatrix}
0 & 0 \\
0 & 1
\end{pmatrix},$
$A_{32} = \begin{pmatrix}
0 & 0 \\
1 & 0
\end{pmatrix},$
$A_{33} = \begin{pmatrix}
0 & 1 \\
0 & 0
\end{pmatrix},$
$A_{34} = \begin{pmatrix}
1 & 0 \\
0 & 0
\end{pmatrix}.$
Then for $R$ we have the following relation graphs:
\begin{center}
\includegraphics[scale = 0.5]{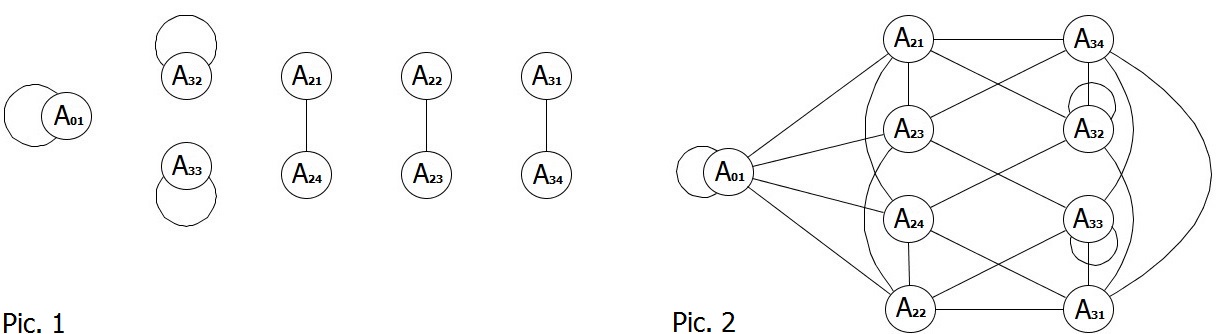}
\includegraphics[scale = 0.5]{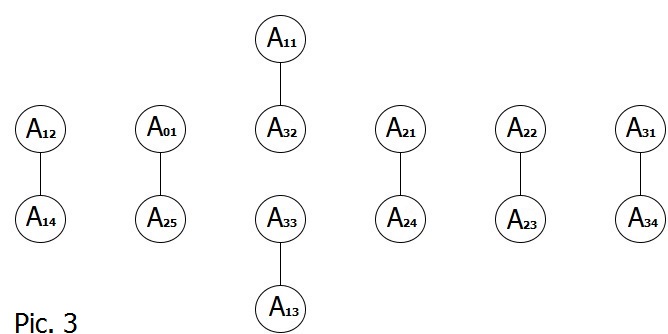}
\end{center}
where Pic.\ 1, Pic.\ 2, and Pic.\ 3 correspond to orthogonality, zero divisor, and commuting graphs, respectively.
\end{ex}

\section{Orthogonality graph of the algebra of upper triangular matrices} \label{S2}

Recall the result from~\cite{BGM} which will be used later.

\begin{lemma}[{\cite[Lemma~4.1]{BGM}}] The orthogonality graph $O\left(M_n\left(\F\right)\right)$  is empty for $n=1$. For $n=2$ the graph $O\left(M_n\left(\F\right)\right)$ is not connected and is the union of its connected subgraphs with the following sets of vertices:

1. The set $$V_1 =
\left\{{\begin{pmatrix}
   a & 0\\
   0 & 0
   \end{pmatrix}\Big| \, 0\ne a \in \F}\right\} \bigcup \left\{{\begin{pmatrix}
   0 & 0\\
   0 & b
   \end{pmatrix}\Big| \, 0\ne b \in \F}\right\};$$

2.  The set $$ V_2 =
\left\{{\begin{pmatrix}
   0 & 0\\
   a & 0
   \end{pmatrix}\Big| \, 0\ne a \in \F}\right\};$$

3.  The set $$V_3=
\left\{{\begin{pmatrix}
   0 & a\\
   0 & 0
   \end{pmatrix}\Big| \, 0\ne a \in \F}\right\};$$

4.  For every $ 0\neq\alpha \in \F$ the set
$$ V_{4,\alpha} =
\left\{{\begin{pmatrix}
   c & c\alpha\\
   0 & 0
   \end{pmatrix}\Big| \, 0\ne c \in \F}\right\} \bigcup \left\{{\begin{pmatrix}
   0 & d\\
   0 & -d/\alpha
   \end{pmatrix}\Big| \, 0\ne d \in \F}\right\};$$

5. For every $ 0\neq\alpha \in \F $  the set $$ V_{5,\alpha} =
\left\{{\begin{pmatrix}
   0 & 0\\
   c & c\alpha
   \end{pmatrix}\Big| \, 0\ne c \in \F}\right\} \bigcup \left\{{\begin{pmatrix}
   d & 0\\
   -d/\alpha & 0
   \end{pmatrix}\Big| \, 0\ne d \in \F}\right\};$$

6. For all $\  0\neq \alpha,\beta \in \F$  the set $$  V_{6,\alpha, \beta} =
\left\{{\begin{pmatrix}
   -\alpha a & a\\
   -\alpha \beta a & \beta a
   \end{pmatrix}\Big| \, 0\ne a \in \F}\right\} \bigcup \left\{{\begin{pmatrix}
   -\beta b & b\\
   -\alpha \beta b & \alpha b
   \end{pmatrix}\Big| \, 0\ne b \in \F}\right\}.$$ \label{lem_comp_m0.1}

The diameter of the connected component corresponding to each of the vertex sets $V_1$, $V_{4,\alpha}$, $V_{5,\alpha}$ equals~$1$ if  $\F=\Z_2$, and equals~$2$ if $|\F|>2$.

The vertex sets $V_{6,\alpha,\beta}$ with $\alpha \ne \beta$ are defined over fields with $|\F|>2$, and the diameters
of the corresponding connected components equal~$2$.

The diameter of the connected component corresponding to any of the vertex sets $V_2$, $V_3$, and $V_{6,\alpha,\alpha}$ equals~$0$ if $\F=\Z_2$, and equals~$1$ if $|\F|>2$.
\end{lemma}

\begin{lemma}
Let $\F$ be a field. Then
the graph $O(T_2)$ is disconnected, and it is a union of its connected subgraphs with the following
sets of vertices:

1. The set $$V_1 =
\left\{{\begin{pmatrix}
   a & 0\\
   0 & 0
   \end{pmatrix}\Big| \, 0\ne a \in \F}\right\} \bigcup \left\{{\begin{pmatrix}
   0 & 0\\
   0 & b
   \end{pmatrix}\Big| \, 0\ne b \in \F}\right\};$$

2.  The set $$V_3=
\left\{{\begin{pmatrix}
   0 & a\\
   0 & 0
   \end{pmatrix}\Big| \, 0\ne a \in \F}\right\};$$ 

3.  For every $ 0\neq\alpha \in \F$  the set
$$ V_{4,\alpha} =
\left\{{\begin{pmatrix}
   c & c\alpha\\
   0 & 0
   \end{pmatrix}\Big| \, 0\ne c \in \F}\right\} \bigcup \left\{{\begin{pmatrix}
   0 & d\\
   0 & -d/\alpha
   \end{pmatrix}\Big| \, 0\ne d \in \F}\right\}.$$

The diameters of the connected components corresponding to all of the vertex sets $V_1$, $V_{4,\alpha}$ equal $1$ if $\F=\Z_2$ and $2$ if $|\F|>2$.
The diameter of the connected component corresponding to the vertex set $V_3$ equals $0$ if $\F=\Z_2$ and $1$ if $|\F|>2$.
\end{lemma}

\begin{proof}
The assertion follows from Lemma~\ref{lem_comp_m0.1} because the graph $O(T_2(\F))$ is a subgraph of
$O(M_2(\F))$ and, for all $0\neq \alpha,\beta\in \F$, the vertex sets $V_1,$ $V_3,$ $V_{4,\alpha}$ belong to $T_2(\F)$,
whereas $V_2,$ $V_{5,\alpha},$ $V_{6,\alpha,\beta}$ contain no upper triangular matrix.
\end{proof}

\begin{lemma}\label{lem_case}
Let $\F$ be a field, $n \ge 3$, and matrices $A, B \in T_n$
have the following block forms:
$$A = \begin{pmatrix}
0 & \bar{a} \\
0_{(n-1) \times 1} & A_1
\end{pmatrix}, \
B = \begin{pmatrix}
B_1 & \bar{b} \\
0_{1 \times (n-1)} & 0
\end{pmatrix},$$
where $A_1, B_1 \in T_{n-1}$ are invertible matrices, $\bar{a} \in M_{1,n-1}(\F),\ \bar{b} \in M_{n-1,1}(\F).$
Then
$$1.\ O_{T_n}(A) = \left\{\begin{pmatrix} c_{0} & \bar{c}\\ 0_{(n-1) \times 1} & 0_{n-1} \end{pmatrix}\, \Big|\, \bar{c}=-c_{0}\bar{a}A_1^{-1},\ c_0 \in \F,\ \bar{c} \in M_{1,n-1}(\F) \right\},$$
$$2.\ O_{T_n}(B) = \left\{\begin{pmatrix} 0_{n-1} & \bar{c} \\ 0_{1 \times (n-1)} & c_{0} \end{pmatrix}\, \Big|\, \bar{c}=-B_1^{-1}\bar{b}c_{0},\ c_0 \in \F,\ \bar{c} \in M_{n-1,1}(\F) \right\}.$$
\end{lemma}

\begin{proof}
1. Assume that $C \in T_n$ has the following block form:
$$C=\begin{pmatrix} c_{0} & \bar{c} \\ 0_{(n-1) \times 1} & C_1 \end{pmatrix},\ C_1\in T_{n-1}.$$
Then $$AC = \begin{pmatrix} 0 & \bar{a}C_1 \\ 0_{(n-1) \times 1} & A_1C_1 \end{pmatrix},\
CA=\begin{pmatrix} 0 & c_{0}\bar{a}+\bar{c}A_1 \\ 0_{(n-1) \times 1} & C_1A_1 \end{pmatrix}.$$
If $C\in O_{T_n}(A)$, $AC=CA=0$. Since $A_1$ is invertible, $C_1=0$ and $\bar{c}=-c_{0}\bar{a}A_1^{-1}.$

2. Assume that $C \in T_n$ has the following block form:
$$C=\begin{pmatrix} C_1 & \bar{c} \\ 0_{1 \times (n-1)} & c_{0} \end{pmatrix},\ C_1\in T_{n-1}.$$
Then $$BC = \begin{pmatrix} B_1C_1 & B_1\bar{c}+\bar{b}c_0 \\ 0_{1 \times (n-1)} & 0 \end{pmatrix},\
CB=\begin{pmatrix} C_1B_1 & C_1\bar{b} \\ 0_{1 \times (n-1)} & 0 \end{pmatrix}.$$
If $C\in O_{T_n}(B)$, $BC=CB=0$. Since $B_1$ is invertible, $C_1=0$ and $\bar{c}=-B_1^{-1}\bar{b}c_{0},$ which completes the proof.
\end{proof}

\begin{defin}
Matrices $A$, $B\in T_n$ satisfying the conditions of the previous lemma will be called {\em bad-1\/} and {\em bad-2\/}, respectively. All the other matrices of $T_n$ are {\em good\/}. A good matrix is called {\em s-good\/} (special good) if it has zero entries in the positions $(1,1)$ and $(n,n)$. A good matrix is called {\em ns-good\/} (not special good) if it is not s-good.
\end{defin}

\begin{rmk}\label{rmk_null}
Note that matrices of $O(T_n)$ have at least one zero entry on their diagonal.
Besides, bad (i.e.\@ bad-1 and bad-2) matrices have only one zero entry that is either $(1,1)$ or $(n,n)$.
Moreover, for each ns-good matrix there exists $i \in \{2, \ldots, n-1\}$ such that its $(i,i)$-entry is zero.
\end{rmk}

\begin{lemma}\label{lem_dist}
For any non-zero singular matrix $A \in T_n$ there exists a rank~$1$ matrix $R_i \in T_n$, corresponding to the zero $(i,i)$-entry of $A$, such that $\dist(A,R_i) \le 1$ in $O(T_n)$.
\end{lemma}

\begin{proof}
By Remark~\ref{rmk_null} the zero $(i,i)$-entry of $A$ exists.
Assume that $A$ has the following block form
$A = \begin{pmatrix}
A_1 & A_2 \\
0_{(n-i) \times i} & A_3
\end{pmatrix}.$
Then $i \times i$ block $A_1$ is singular.
Therefore, there exists a non-zero vector $\hat x \in \F^i$ such that $A_1 \hat x=0$.
Hence $A$ annihilates a vector $\bar{x}=\hat x \oplus 0_{n-i}$.
Likewise we argue for transpose of $A$ which annihilates a non-zero vector $\bar{f}=0_{i-1} \oplus \hat f$.
Then $R_i=\bar{x}\bar{f}^t$ is the desired upper triangular matrix of rank~$1$, since
$AR_i=\left(A\bar{x}\right)\bar{f}^t=0=\bar{x}\left(A^t\bar{f}\right)^t=R_iA$.
\end{proof}

\begin{rmk}\label{rmk_1n}
If $i \ne 1$, then the first column of $R_i$ is zero, and if $i \ne n$, then the last row of $R_i$ is zero.
\end{rmk}

\begin{theorem}\label{Th_diam_Tr}
Let $\F$ be a field and $n \ge 3$.
Then the graph $O(T_n)$ is connected and $\diam O(T_n)=4$.
\end{theorem}

\begin{proof}
Let $A$, $B\in T_n$ be two non-zero singular matrices.
We show that  there is~a path between $A$ and $B$  in $O(T_n)$ of length at most~$4$.
The general situation splits into the 3 following cases.

1. First, assume that both $A$ and $B$ are good.
We have the following subcases.

1.1. Suppose that both $A$ and $B$ are s-good.
Since the last row and the first column of both matrices are zero, we have the path
$$A - E_{1n} - B.$$

1.2. Suppose that both $A$ and $B$ are ns-good.
Then by Remark~\ref{rmk_null} there exist $i,j~\in \{2, \ldots, n-1\}$ such that $(i,i)$-entry of $A$ and $(j,j)$-entry of $B$ are zero.
By Lemma~\ref{lem_dist} we can find non-zero upper triangular rank~$1$ matrices $R_i=\bar{x}\bar{f}^t$ and $R_j=\bar{y}\bar{h}^t$ corresponding to the zero $(i,i)$-entry of $A$ and the zero $(j,j)$-entry of $B$, respectively, with $\dist(A,R_i) \le 1$, $\dist(B,R_j) \le 1$ in $O(T_n)$.
Since $1 < i,j < n$, by Remark~\ref{rmk_1n} the last row and the first column of both $R_i$ and $R_j$ are zero.
Hence we have the path
$$A - R_i - E_{1n} - R_j - B.$$

1.3. Without loss of generality suppose that $A$ is s-good and $B$ is ns-good.
As in item 1.2, we can find the matrix $R_j$ that is orthogonal to $B$ and $E_{1n}$.
Hence we have the path
$$A - E_{1n} - R_j - B.$$

2. Consider the case when both $A$ and $B$ are bad. We have the following subcases.

2.1. Suppose that both $A$ and $B$ are bad-1.
By Lemma~\ref{lem_case} we can find non-zero matrices
$$A_1 = \begin{pmatrix} a_{1_{0}} & \bar{a_1}\\ 0_{(n-1) \times 1} & 0_{n-1} \end{pmatrix},\
B_1 = \begin{pmatrix} b_{1_{0}} & \bar{b_1}\\ 0_{(n-1) \times 1} & 0_{n-1} \end{pmatrix}$$
that are orthogonal to $A$ and $B$, respectively.
Since $n \ge 3$, there exists a non-zero element $\bar{d} \in \F^{n}$ such that $A_1\bar{d}=B_1\bar{d}=0$.
If $D$ is a matrix with first $n-1$ zero columns and its last column is $\bar{d}$, then we have the path
$$A - A_1 - D - B_1 - B.$$

2.2. Suppose that both $A$ and $B$ are bad-2.
By Lemma~\ref{lem_case} we can find non-zero matrices
$$A_2 = \begin{pmatrix} 0_{n-1} & \bar{a_2}\\ 0_{1 \times (n-1)} & a_{2_{0}} \end{pmatrix},\
B_2 = \begin{pmatrix} 0_{n-1} & \bar{b_2}\\ 0_{1 \times (n-1)} & b_{2_{0}} \end{pmatrix}$$
that are orthogonal to $A$ and $B$, respectively.
Since $n \ge 3$, there exists a non-zero element $\bar{d'} \in M_{1,n}(\F)$ such that $\bar{d'}A_7=\bar{d'}B_5=0$.
If $D'$ is a matrix whose first row is $\bar{d'}$ and the other rows are zero, then we have the path
$$A - A_2 - D' - B_2 - B.$$

2.3. Without loss of generality suppose that $A$ is bad-1 and $B$ is bad-2.
As in items 2.1 and 2.2, we can find non-zero matrices
$$A_1 = \begin{pmatrix} a_{1_{0}} & \bar{a_1}\\ 0_{(n-1) \times 1} & 0_{n-1} \end{pmatrix},\
B_2 = \begin{pmatrix} 0_{n-1} & \bar{b_2}\\ 0_{1 \times (n-1)} & b_{2_{0}} \end{pmatrix}$$
that are orthogonal to $A$ and $B$, respectively.
Let the first entry of the vector $\bar{a_1}$ be $a_{1_{1}}$, and let the first two entries of the vector $\bar{b_2}$ be $b_{2_{1}}$ and $b_{2_{2}}$.
There exists a non-zero matrix $L' \in M_2(\F)$ such that
$\begin{pmatrix} a_{1_{0}} & a_{1_{1}} \end{pmatrix}L'=L'\begin{pmatrix} b_{2_{1}} \\ b_{2_{2}} \end{pmatrix}=0$.
Then if we put $L~=~\begin{pmatrix}0_{2\times(n-2)}&L'\\ 0_{(n-2)\times (n-2)}&0_{(n-2)\times 2}\end{pmatrix}$,
we have the path $$A - A_1 - L - B_2 - B.$$

3. Now without loss of generality suppose that $A$ is good and $B$ is bad. We have the following subcases.

3.1. Suppose that $A$ is ns-good and $B$ is bad-1.
As in items 1.2 and 2.1, we can find non-zero matrices
$$R_i=\bar{x}\bar{f}^t = \begin{pmatrix} 0_{(n-1) \times 1} & R_i'\\ 0 & 0_{1 \times (n-1)} \end{pmatrix},\
B_1 = \begin{pmatrix} b_{1_{0}} & \bar{b_1}\\ 0_{(n-1) \times 1} & 0_{n-1} \end{pmatrix}$$
that are orthogonal to $A$ and $B$, respectively.
Since $R_i'$ is singular, there exists a non-zero element $\bar{c} \in \F^{n-1}$ such that $R_i'\bar{c}=0$.
If we put
$C = \begin{pmatrix} 0_{1 \times (n-1)} & c_0 & \\ 0_{n-1} & \bar{c} \end{pmatrix},$
where $c_0 = -\bar{b_1}\bar{c}/b_{1_0},$
we have the path $$A - R_i - C - B_1 - B.$$

3.2. Suppose that $A$ is s-good and $B$ is bad-1.
Since the matrix $A$ is orthogonal to $E_{1n}=\begin{pmatrix} 0_{(n-1) \times 1} & E_{1n}'\\ 0 & 0_{1 \times (n-1)} \end{pmatrix}$,
where $E_{1n}'=E_{1,n-1} \in M_{n-1}(\F) \setminus GL_{n-1}(\F)$, then, as in item~3.1, we can find the path $$A - E_{1n} - C' - B_1 - B.$$

3.3. Suppose that $A$ is ns-good and $B$ is bad-2.
As in items 1.2 and 2.2, we can find non-zero matrices
$$R_i=\bar{x}\bar{f}^t = \begin{pmatrix} 0_{(n-1) \times 1} & R_i'\\ 0 & 0_{1 \times (n-1)} \end{pmatrix},\
B_2 = \begin{pmatrix} 0_{n-1} & \bar{b_2}\\ 0_{1 \times (n-1)} & b_{2_{0}} \end{pmatrix}$$
that are orthogonal to $A$ and $B$, respectively.
Since $R_i'$ is singular, there exists a non-zero element $\bar{c_1} \in M_{1,n-1}(\F)$ such that $\bar{c_1}R_i'=0$.
If we put
$C_1 = \begin{pmatrix} \bar{c_1} & c_{1_0} & \\ 0_{n-1} & 0_{(n-1) \times 1} \end{pmatrix},$
where $c_{1_0} = -\bar{c_1}\bar{b_2}/b_{2_0},$
we have the path $$A - R_i - C_1 - B_2 - B.$$

3.4. Suppose that $A$ is s-good and $B$ is bad-2.
Since the matrix $A$ is orthogonal to $E_{1n}=\begin{pmatrix} 0_{(n-1) \times 1} & E_{1n}'\\ 0 & 0_{1 \times (n-1)} \end{pmatrix}$,
where $E_{1n}'=E_{1,n-1} \in M_{n-1}(\F) \setminus GL_{n-1}(\F)$,
then, as in item~3.3, we can find the path $$A - E_{1n} - C'' - B_2 - B.$$

Thus we have shown that that $\diam O(T_n) \le 4$ in all cases as desired.

Now we claim that $\diam O(T_n) = 4$.
Let $\hat A = I_n-E_{11}$ and
$\hat B = J_n.$
Straightforward computations show that
$$\ O_{T_n}(\hat A) = \left\{\alpha E_{11} \, \Big|\, \alpha \in \F \right\},\
O_{T_n}(\hat B) = \left\{\alpha E_{1n} \, \Big|\, \alpha \in \F\right\}.$$
Clearly, $E_{11}E_{1n} \neq 0$, hence $\dist(\hat A,\hat B)>3$ and the proof is completed.
\end{proof}

\section*{Acknowledgements}

The author is grateful to his scientific supervisor Alexander E. Guterman for posing this problem and fruitful discussions.

\end{document}